\documentclass[oneside,english]{amsart}
\usepackage[T1]{fontenc}
\usepackage[latin9]{inputenc}
\usepackage{geometry}
\usepackage{amstext}
\usepackage{amsthm}
\usepackage{amssymb}
\usepackage{hyperref}

\makeatletter
\numberwithin{equation}{section}
\numberwithin{figure}{section}
\theoremstyle{plain}
\newtheorem{thm}{\protect\theoremname}
\theoremstyle{plain}
\newtheorem{defn}{\protect\defnname}
\theoremstyle{plain}
\newtheorem{cor}[thm]{\protect\corollaryname}
\theoremstyle{remark}
\newtheorem{rem}[thm]{\protect\remarkname}
\theoremstyle{plain}
\newtheorem{conj}{\protect\conjname}

\newcommand{\N}{\mathcal{N}}
\newcommand{\Z}{\mathcal{Z}}
\newcommand{\X}{\mathcal{X}}

\DeclareMathOperator{\var}{Var}
\DeclareMathOperator{\sd}{sd}
\DeclareMathOperator{\sn}{sn}

\DeclareMathOperator{\scellip}{sc}
\DeclareMathOperator{\slellip}{sl}

\makeatother

\usepackage{babel}
\providecommand{\corollaryname}{Corollary}
\providecommand{\remarkname}{Remark}
\providecommand{\defnname}{Definition}
\providecommand{\theoremname}{Theorem}
\providecommand{\conjname}{Conjecture}

\begin{document}
\title{Taylor coefficients of the Jacobi $\theta_{3}\left( q \right)$ function}
\author{Tanay Wakhare and Christophe Vignat}
\address{University of Maryland, College Park, MD 20742, USA}
\email{twakhare@gmail.com}

\address{L.S.S., CentraleSup\'{e}lec, Universit\'{e} Paris Sud, Orsay, France and Department of Mathematics, Tulane University, New Orleans, USA}
\email{cvignat@tulane.edu}

\begin{abstract}
We extend some results recently obtained by Dan Romik \cite{Romik}
about the Taylor coefficients of the theta function $\theta_{3}\left(e^{-\pi}\right)$
to the case $\theta_{3}\left(q\right)$ of a real valued variable $0<q<1$. These results are obtained by carefully studying 
the properties of the cumulants associated to a Jacobi $\theta_{3}$ (or
discrete normal) distributed random variable. This article also states
some integrality conjectures about rational sequences that generalize
the one studied by Romik.
\end{abstract}

\maketitle

\section{Introduction}

Recently, Dan Romik derived \cite{Romik} some results about the Taylor
coefficients of the elliptic Jacobi $\theta_{3}$ function, which is defined as 
$$\theta_{3}\left(q\right):=\sum_{n= -\infty}^{\infty}q^{n^{2}}.$$
One of Romik's most
unexpected results is the existence of an integer sequence $d(n), n \geq 0,$
\[
1,1,-1,51,849,-26199,1341999,82018251\dots,
\]
which allows 
to compute, in the standard case $q=e^{-\pi},$
any moment of a discrete normal distribution
as a finite sum according to the formula \footnote{we have rewritten this in terms of our notation for $\theta_3$} \cite[Thm 1]{Romik}
\begin{equation}\label{1.1}
\frac{1}{\theta_{3}\left(e^{ - \pi }\right)}\sum_{p=-\infty}^{+\infty}p^{2n}e^{-\pi p^{2}}=\frac{1}{\left(4\pi\right)^{n}}\sum_{j=0}^{\left\lfloor \frac{n}{2}\right\rfloor }\frac{\left(2n\right)!}{2^{n-2j}\left(4j\right)!\left(n-2j\right)!}d\left(j\right)\Omega^{j},
\end{equation}
with the constant $\Omega:=\frac{\Gamma^{8}\left(\frac{1}{4}\right)}{32\pi^{8}}.$

One unanswered question in Romik's paper is the extension of this
result to an arbitrary real parameter $q \in (0,1)$.
The main result of this article is our Theorem \ref{thm:moment theta3}, that for an arbitrary value of $k\in \left( 0,1 \right),$
\[
\frac{1}{\theta_{3}\left(q\right)}\sum_{p=-\infty}^{+\infty}p^{2n}q^{p^{2}}=\sum_{j=0}^{n}\binom{2n}{2j}\left(\frac{z}{2}\right)^{2j}R_{2j}\left(k\right)\left(-\frac{1}{2}\right)^{n-j}\sigma^{2n-2j}\frac{\left(2n-2j\right)!}{\left(n-j\right)!}
\]
with $q=e^{-\pi\frac{K'\left(k\right)}{K\left(k\right)}},$ $\sigma^{2}=\frac{K\left(k\right)^{2}}{\pi^{2}}\left[\frac{E\left(k\right)}{K\left(k\right)}-\left(k'\right)^{2}\right]$, $\left\{ R_{2j}\left(k\right)\right\} $ a sequence of polynomials
described in Definition \ref{Rdef}, and $E(K)$, $K(k)$, and $K'(k)$ elliptic integrals defined in \eqref{ellipint}. The case $k=\frac{1}{\sqrt{2}}$
reduces to Romik's identity \eqref{1.1}.

This result can be interpreted as the calculation of the moments of a discretized Gaussian distribution. In this article we adopt a probabilistic language, although it is not necessary in this context; however, it simplifies the statement
of most results. Let us denote by $\X_{\theta_{3}}$ a \textbf{discrete normal random variable}
with parameter $k:$ 
\[
\Pr\left\{ \X_{\theta_{3}}=n\right\} =\frac{1}{\theta_{3}\left(q\right)}q^{n^{2}},\thinspace\thinspace n\in\mathbb{Z},
\]
with $q=e^{-\pi\frac{K'\left(k\right)}{K\left(k\right)}}.$ 
This distribution is essentially the (continuous) standard normal distribution sampled at integer points, then rescaled.
Its moment
generating function is defined by
\[
\varphi_{\theta_{3}}\left(u\right)=\mathbb{E}e^{u\X_{\theta_{3}}}=\frac{1}{\theta_{3}\left(q\right)}\sum_{n= -\infty}^{\infty}e^{nu}q^{n^{2}}
\]
and its cumulants $\kappa_{n}$ are defined as the Taylor coefficients
of the cumulant generating function
\[
\psi_{\theta_{3}}\left(u\right)=\log\varphi_{\theta_{3}}\left(u\right)=\sum_{n=1}^{\infty}\kappa_{n}\frac{u^{n}}{n!}.
\]

To obtain our main results, we extensively study the cumulants of $\X_{\theta_{3}}$
and show that they have a rich structure: not only they are related
with Eisenstein and Lambert series,
but they also possess a clean combinatorial interpretation in terms of restricted permutations. 

The most basic result we need is Theorem \ref{cumulants P}, which re-expresses the cumulants in terms of Schett polynomials. This is then used to prove Theorems \ref{thm:sequence d(n)} and \ref{thm:moment theta3}, which are equivalent rephrasings of the same result and which generalize Romik's identity \eqref{1.1}. Finally, this is used to prove Corollary \ref{corollary}, which provides a new expression for Romik's sequence $d(n)$. The remainder of this paper consists of new expansions for the moments and cumulants of a $\theta_3$ random variable, which we believe will be useful in the further study of $\X_{\theta_3}$. In the future, it would be nice to see similar studies of the other Jacobi theta functions $\theta_1, \theta_2$ and $\theta_4$, as we believe similar results should hold for them.

Throughout this paper, given a random variable $\Z$, $\mu_n\left(\Z\right)$ will denote its $n$-th moment, and $\kappa_n\left(\Z\right)$ its $n$-th cumulant. If there is no parametrization and we write $\mu_n$ or $\kappa_n$, it is implicit that these are the moments and cumulants of a $\theta_3$ distributed random variable.

Romik's sequence $d(n)$ can alternately be described as the 
sequence of Taylor coefficients of the centered theta function
$$\frac{1}{\sqrt{1+z}}\theta_3\left( e^{ \pi \frac{z-1}{z+1}}  \right)  = \theta_3(e^{-\pi})  \sum_{n=0}^\infty \frac{d(n)}{(2n)!}  \left( \frac{\Gamma^8\left(\frac14\right)}{2^7\pi^4}\right)^nz^{2n}. $$
The study of Taylor coefficients of modular forms and their arithmetic properties has a long history, much of which is surveyed in \cite[Section 6]{Zagier123}. At complex multiplication points, the Taylor coefficients of modular functions with algebraic Fourier coefficients are known to be algebraic; the magic here is that for the special case of the Jacobi theta function, up to a power of the transcendental factor $\Omega$, the coefficients are in fact given by explicitly computable integers. This naturally suggests the study of the arithmetic and combinatorial properties of $d(n)$, seeing as other integer sequences arising from the study of modular functions tend to have nice arithmetic properties; for instance, the norm of the difference of two singular moduli factors nicely \cite[p. 78]{Zagier123}, and the central values of $L$-series of higher weight grossencharacters are attached to integers which are in fact squares \cite[p. 95]{Zagier123}. A series of recent papers \cite{Wakhare, Scherer} have addressed arithmetic aspects of $d(n)$, showing that modulo any prime $d(n)$ is either periodic or vanishes, while \cite{Rolen} proved that the Taylor coefficients of half integral weight modular forms at CM points are periodic modulo split primes. This paper contributes to the wider study of Taylor expansions of modular functions by providing new expressions for the Taylor coefficients of the classical $\theta_3$ function and its logarithm.

In the special case of the theta function, the study of the Taylor coefficients also connects to the study of zeros of the Riemann zeta function due to the integral identity \cite{Romik}
$$ \pi^{-\frac{s}{2}} \Gamma\left(\frac{s}{2}\right)\zeta(s) = \int_{-1}^1  \left(     \frac{1}{\sqrt{1+t}}\theta_3\left( e^{ \pi \frac{t-1}{t+1}}  \right) - \frac{1}{\sqrt{1+t}} \right) (1-t)^{\frac{s}{2}-1} (1+t)^{\frac{(1-s)}{2}-1}dt,$$
the sort of expansion studied by Tur\'an to estimate the distance of zeta zeros from the origin. Tur\'an then instead studied expansions with respect to the Hermite polynomials, a methodology later explored by Romik \cite{Romik2}. There is still hope that new information about the Taylor coefficients $d(n)$, their sign patterns, and their $p$-adic properties will yield information about zeta zeros through such integral identities.

\section{Introduction to elliptic functions}

Our approach relies heavily on properties of elliptic functions. We
recall here some basic results and notations about those used in this
article. Some useful references about elliptic functions are the classic
\cite{Apostol,Neville} and the more recent \cite{Armitage}. Additionally,
\cite{Roy} provides an extensive historical approach.

For a given elliptic modulus $k\in \left( 0,1 \right)$, the complete elliptic integrals
of the first and second kind are the functions
\begin{equation}\label{ellipint}
K\left(k\right)=\int_{0}^{1}\frac{dt}{\sqrt{1-t^{2}}\sqrt{1-k^{2}t^{2}}},\thinspace\thinspace E\left(k\right)=\int_{0}^{1}\frac{\sqrt{1-k^{2}t^{2}}}{\sqrt{1-t^{2}}}dt.
\end{equation}
Expansion of the factors $\frac{1}{\sqrt{1-k^{2}t^{2}}}$ and $\sqrt{1-k^{2}t^{2}}$
in the respective integrals shows that these integrals can also be expressed as hypergeometric
functions according to \cite[(3.2.3), (3.2.14)]{Andrews}
\[
K\left(k\right)=\frac{\pi}{2}\thinspace_{2}F_{1}\left(\begin{array}{c}
\frac{1}{2},\frac{1}{2}\\
1
\end{array},k^{2}\right),\thinspace\thinspace E\left(k\right)=\frac{\pi}{2}\thinspace_{2}F_{1}\left(\begin{array}{c}
\frac{1}{2},-\frac{1}{2}\\
1
\end{array},k^{2}\right).
\]
The complementary elliptic modulus is
\[
k'=\sqrt{1-k^{2}}
\]
and we adopt the usual notations
\[
E'\left(k\right)=E\left(k'\right)=E\left(\sqrt{1-k^{2}}\right),\thinspace\thinspace K'\left(k\right)=K\left(k'\right)=K\left(\sqrt{1-k^{2}}\right).
\]
The nome $q$ is defined as a function of the elliptic modulus $k$
as
\[
q=e^{-\pi\frac{K'\left(k\right)}{K\left(k\right)}}.
\]
We will frequently refer to the \textbf{standard case} $k = \frac{1}{\sqrt{2}}$. Substituting this value of $k$ then gives the complementary elliptic modulus $k' = \sqrt{1 - k^2} = \frac{1}{\sqrt{2}} = k$ and nome $q = e^{-\pi}$. This is the value of $k$ for which our results reduce to those of Romik.
We record here the values \cite[(2.47)]{Armitage}
\begin{equation}\label{kspecialval}
K\left( \frac{1}{\sqrt{ 2 }} \right) = \frac{\Gamma^2\left( \frac{1}{4} \right)}{4\sqrt{\pi}},\,\,
E\left(\frac{1}{\sqrt{ 2 }} \right) = \frac{4 \Gamma^2\left( \frac{3}{4} \right)+\Gamma^2\left( \frac{1}{4} \right)}{8\sqrt{\pi}}.
\end{equation}
The Jacobi theta functions are
\[
\theta_{1}\left(z,q\right):=\sum_{n= -\infty}^{\infty}\left(-1\right)^{n-\frac{1}{2}}q^{\left(n+\frac{1}{2}\right)^{2}}e^{ i \left(2n+1\right)z},\thinspace\thinspace  \theta_{2}\left(z,q\right):=\sum_{n= -\infty}^{\infty}q^{\left(n+\frac{1}{2}\right)^{2}}e^{ i \left(2n+1\right)z}
\]
and
\[
\theta_{3}\left(z,q\right):=\sum_{n= -\infty}^{\infty}q^{n^{2}}e^{ i 2nz},\thinspace\thinspace \theta_{4}\left(z,q\right):=\sum_{n= -\infty}^{\infty}\left(-1\right)^{n}q^{n^{2}}e^{ i 2nz},
\]
and we adopt the shortcut notation 
\[
\theta_{i}\left(q\right):=\theta_{i}\left(0,q\right),\thinspace\thinspace1\le i\le4.
\]
They have infinite product representations \cite[(10.7.7)]{Andrews}
\begin{equation}
\theta_{1}\left(z,q\right)=- i  q^{\frac{1}{4}}e^{ i  z}\left(q^{2},q^{2}\right)_{\infty}\left(q^{2}e^{ i 2z},q^{2}\right)_{\infty}\left(q^{2}e^{- i 2z},q^{2}\right)_{\infty},
\end{equation}
\begin{equation}
\theta_{2}\left(z,q\right)=q^{\frac{1}{4}}e^{ i  z}\left(q^{2},q^{2}\right)_{\infty}\left(-q^{2}e^{ i 2z},q^{2}\right)_{\infty}\left(-q^{2}e^{- i 2z},q^{2}\right)_{\infty},\label{eq:theta2 infinite product}
\end{equation}
\begin{equation}
\theta_{3}\left(z,q\right)=\left(q^{2},q^{2}\right)_{\infty}\left(-qe^{ i 2z},q^{2}\right)_{\infty}\left(-qe^{- i 2z},q^{2}\right)_{\infty},\label{eq:theta3 infinite product}
\end{equation}
\begin{equation}
\theta_{4}\left(z,q\right)=\left(q^{2},q^{2}\right)_{\infty}\left(qe^{ i 2z},q^{2}\right)_{\infty}\left(qe^{- i 2z},q^{2}\right)_{\infty}.
\end{equation}
The theta functions are related to the complete elliptic integral
of the first kind by
\begin{equation}
\theta_{2}^{2}\left(q\right)=\frac{2}{\pi}kK\left(k\right),\thinspace\thinspace\theta_{3}^{2}\left(q\right)=\frac{2}{\pi}K\left(k\right),\thinspace\thinspace\theta_{4}^{2}\left(q\right)=\frac{2}{\pi}k'K\left(k\right).\label{eq:theta K}
\end{equation}
Jacobi's identity expresses the transformation of 
the $\theta_3$ function under the action of the modular group: denoting 
\[
q=e^{ i \pi\tau},\thinspace\thinspace\tau= i \frac{K'\left(k\right)}{K\left(k\right)},
\]
the invariance reads
\[
\theta_{3}\left(-\frac{1}{\tau}\right)=\sqrt{- i \tau}\theta_{3}\left(\tau\right)
\]
or, expressed in terms of the elliptic modulus,
\[
\frac{1}{\sqrt{K\left(k\right)}}\theta_{3}\left(e^{-\pi\frac{K'\left(k\right)}{K\left(k\right)}}\right)=\frac{1}{\sqrt{K\left(k'\right)}}\theta_{3}\left(e^{-\pi\frac{K\left(k\right)}{K'\left(k\right)}}\right),
\]
an identity that can be interpreted as the invariance of $\theta_{3}$
under the change of parameter $k\mapsto k' = \sqrt{1-k^2}.$ In the parameterization
\[
\theta_{3}\left(e^{-\pi c}\right)=\sum_{n= -\infty}^{\infty}e^{-\pi n^{2}c},
\]
with $c=\frac{K'\left(k\right)}{K\left(k\right)},$ this invariance
reads
\[
\theta_{3}\left(e^{-\pi\frac{1}{c}}\right)=\sqrt{c}\theta_{3}\left(e^{-\pi c}\right).
\]
The Eisenstein series $G_{2k}\left(\tau\right)$ of weight $2k$ with
$k\ge2$ are defined as
\[
G_{2k}\left(\tau\right)=\sum_{\left(m,n\right)\ne\left(0,0\right)}\frac{1}{\left(m+\tau n\right)^{2k}}.
\]
The Weierstrass $\wp$ elliptic function with periods $\omega_1$ and $\omega_2$ is defined as
\[
\wp\left(z;\omega_{1},\omega_{2}\right)=\frac{1}{z^{2}}+\sum_{\left(m,n\right)\ne\left(0,0\right)}  \left[ \frac{1}{\left(z+m\omega_{1}+n\omega_{2}\right)^{2}}-\frac{1}{\left(m\omega_{1}+n\omega_{2}\right)^{2}} \right],
\]
where $\omega_1,\omega_2 \in \mathbb{C}\setminus\{0\}$ and $\frac{\omega_1}{\omega_2} \not \in \mathbb{R}$. T
With $\tau=\frac{\omega_2}{\omega_1}$, the invariants $g_{2}$ and $g_{3}$ are defined by
\[
g_{2}=\frac{60}{\omega_1^4}G_{4}\left( \tau \right),\thinspace\thinspace g_{3}=\frac{140}{\omega_1^6}G_{6}\left( \tau \right).
\]

\section{Further Definitions}
There are three sequences of polynomials which are used throughout this work. We have extracted all the relevant definitions here.
\begin{defn}\label{Xdef}
The \textbf{Schett polynomials} $X_{n}\left(x,y,z\right)$ are defined by the
recurrence \cite{Dumont 2}
\[
X_{n}=\left(yz\frac{d}{dx}+zx\frac{d}{dy}+xy\frac{d}{dz}\right)X_{n-1},
\]
with initial condition 
\[
X_{0}\left(x,y,z\right)=x.
\]
\end{defn}
The first values are
\begin{align*}
&X_{0}\left(x,y,z\right)=x,\\
&X_{1}\left(x,y,z\right)=yz,\\
&X_{2}\left(x,y,z\right)=x\left(y^{2}+z^{2}\right),\\
&X_{3}\left(x,y,z\right)=yz\left(y^{2}+z^{2}+4x^{2}\right),\\
&X_{4}\left(x,y,z\right)=x\left(y^{4}+z^{4}\right)+4x^{3}\left(y^{2}+z^{3}\right)+14xy^{2}z^{2}.
\end{align*}
We will be interested in the special cases $X_{2n+1}\left(0,k, i  k'\right)$, with $k'=\sqrt{1-k^2},$ the complementary elliptic modulus. These have first values
\begin{align*}
&X_{1}\left(0,k, i  k'\right)= i  kk',\\ 
&X_{3}\left(0,k, i  k'\right)= i  kk'\left(2k^{2}-1\right),\\ 
&X_{5}\left(0,k, i  k'\right)= i  kk'\left(k^{4}-14k^{2}\left(k'\right)^{2}+\left(k'\right)^{4}\right).
\end{align*}
The Schett polynomials, and their convolutions $P_n$ appear naturally in the expression of the cumulants in Theorem \ref{cumulants P}.
\begin{defn}\label{Pdef}
We define the \textbf{self-convolution of the Schett polynomials} $P_{2n}(k)$ by
\[
P_{2n}\left(k\right):=\sum_{p=0}^{n-1}\binom{2n}{2p+1}X_{2p+1}\left(0,k, i  k'\right)X_{2n-2p-1}\left(0,k, i  k'\right),\thinspace\thinspace n\ge1,
\]
\end{defn}
The first values are
\begin{align*}
&P_{0}\left(k\right)=0,\\
&P_{2}\left(k\right)=-2\left(kk'\right)^{2}, \\
&P_{4}\left(k\right)=-8\left(kk'\right)^{2}\left(2k^{2}-1\right), \\
&P_{6}\left(k\right)=-16\left(kk'\right)^{2}\left(2-17k^{2}+17k^{4}\right).
\end{align*}

Notice that in what follows, we are studying a random variable $\Z$ whose real part is a discrete Gaussian, and whose imaginary part is a continuous Gaussian. 
\begin{defn}\label{Rdef}
Consider the random variable 
\[
\Z=\X_{\theta_{3}}+ i  \N_{\sigma^{2}}
\]
where $\N_{\sigma^{2}}$ is a (continuous) Gaussian random variable whose variance
$\sigma^{2}=\kappa_{2}$ coincides with the variance of the discrete Gaussian random variable $\X_{\theta_{3}},$ and $\N_{\sigma^2}$ and $\X_{\theta_3}$ are independent. Furthermore, let $z=\theta_{3}^{2}\left(q\right)=\frac{2}{\pi}K\left(k\right)$. Then we define a sequence of \textbf{moment polynomials} $R_{2n}(k)$ by
$$ R_{2n}(k) := \left(\frac{2}{z}\right)^{2n}\mu_{2n}(\Z).$$
\end{defn}
The first few cases are
\begin{align*}
&R_{2}\left(k\right)=0,\\
&R_{4}\left(k\right)=2\left(kk'\right)^{2},\\
& R_{6}\left(k\right)=-8\left(kk'\right)^{2}\left(1-2k^{2}\right),\\
&R_{8}\left(k\right)=4\left(kk'\right)^{2}\left(8-33k^{2}+33k^{4}\right), \\
&R_{10}\left(k\right)=32\left(kk'\right)^{2}\left(4-27k^{2}+57k^{4}-38k^{6}\right), \\
&R_{12}\left(k\right)=8\left(kk'\right)^{2}\left(64-632k^{2}+2187k^{4}-3110k^{6}+1555k^{8}\right).
\end{align*}

We defer the justification that this is in fact a polynomial in $k$ and $k'$ of degree $2n$ till the proof of Theorem \ref{thm:sequence d(n)}, as it is somewhat involved. We also want to note exactly how this is parametrized by $k$. We first pick a value $k \in (0,1)$, which then yields the complementary elliptic modulus $k' = \sqrt{1-k^2}$. This then gives the nome $$q = \exp \left( - \pi\frac{K'(k)}{K(k)}\right),$$ which is implicit in the discrete pdf
$$\Pr\left\{ \X_{\theta_{3}}=n\right\} =\frac{1}{\theta_{3}\left(q\right)}q^{n^{2}},$$
which in turn affects the moments $\mu_{2n}(\Z)$.

\section{Main results}
Our main result is an extension of Romik's identity (\ref{1.1})
as follows. Let us introduce the Hermite polynomials $H_n\left( x \right)$ defined by the generating function
\[
\sum_{n=0}^{\infty} H_{n}\left( x \right) \frac{w^n}{n!} = e^{2xw-w^2}.
\]
Throughout this section, we will frequently refer to the variance
\begin{equation}\label{eq:cumulants}
\sigma^{2}=\frac{K\left(k\right)^{2}}{\pi^{2}}\left[\frac{E\left(k\right)}{K\left(k\right)}-\left(k'\right)^{2}\right].
\end{equation}
In Section \ref{subsec:The-cumulant-generating}, we will show that $\sigma^2 = \var \X_{\theta_3}$. Furthermore, in the standard case $k = \frac{1}{\sqrt{2}}$, note that by the special values \eqref{kspecialval} this variance reduces to $\sigma^2=\frac{1}{4\pi}$.
\begin{thm}
\label{thm:sequence d(n)}With $H_n\left( x \right)$ denoting the $n$--th Hermite polynomial, we have
\begin{equation}
\frac{1}{\theta_{3}\left(q\right)}\sum_{p=-\infty}^{+\infty}q^{p^{2}}H_{2n}\left(\frac{p}{\sigma\sqrt{2}}\right)=\left(\frac{z^{2}}{2\sigma^{2}}\right)^{n}R_{2n}\left(k\right),\label{eq:extProp10}
\end{equation}
where $\sigma^2$ is given by (\ref{eq:cumulants}), $R_{2n}\left(k\right)$ is defined in Definition \ref{Rdef} and  $z=\theta_{3}^{2}\left(q\right)=\frac{2}{\pi}K\left(k\right)$.
\end{thm}

\begin{proof}
Consider the random variable 
\[
\Z=\X_{\theta_{3}}+ i  \N_{\sigma^{2}}
\]
from the setup of Definition \ref{Rdef}. Recall that $\N_{\sigma^{2}}$ is a Gaussian random variable whose variance
$\sigma^{2}=\kappa_{2}$ coincides with the variance of $\X_{\theta_{3}},$ and $\N_{\sigma^2}$ and $\X_{\theta_3}$ are independent. Then
\[
\kappa_{2}\left(\Z\right)=\kappa_{2}\left(\X_{\theta_{3}}\right)-\kappa_{2}\left(\N_{\sigma^{2}}\right)=0
\]
while, for $n\ge2,$
\[
\kappa_{2n}\left(\Z\right)=\kappa_{2n}\left(\X_{\theta_{3}}\right)+ i ^{2n}\kappa_{2n}\left(\N_{\sigma^{2}}\right)=\kappa_{2n}\left(\X_{\theta_{3}}\right),
\] since the cumulants of order $2n$ of a continuous Gaussian random variable are all zero for $n \ge 2.$
Moreover, $\Z$ has odd moments equal to $0$ as both $\X_{\theta_3}$ and $\N_{\sigma^2}$ are symmetric about the origin. The even moments of $\Z$ are given by
\begin{align*}
\mu_{2n}\left(\Z\right) & =\mathbb{E}\left(\X_{\theta_{3}}+ i  \N_{\sigma^{2}}\right)^{2n}=\frac{1}{\theta_{3}\left(q\right)}\sum_{p=-\infty}^{+\infty}
q^{p^2}
\mathbb{E}\left(p+ i  \N_{\sigma^{2}}\right)^{2n}\\
 & =\frac{1}{\theta_{3}\left(q\right)}\sum_{p=-\infty}^{+\infty}
q^{p^2}
\mathbb{E}\left(p+ i \sigma\sqrt{2}\N_{\frac{1}{2}}\right)^{2n}\\
 & =\frac{1}{\theta_{3}\left(q\right)}\left(\frac{\sigma}{\sqrt{2}}\right)^{2n}\sum_{p=-\infty}^{+\infty}
q^{p^2}
H_{2n}\left(\frac{p}{\sigma\sqrt{2}}\right),
\end{align*}
where we have used the representation for the Hermite polynomials \cite[8.951]{Gradshteyn}
\[
H_{2n}\left(w\right)=2^{2n}\mathbb{E}\left(w+ i  \N_{\frac{1}{2}}\right)^{2n}.
\]
On the other hand, the moments of $\Z$ can be expressed in terms of
complete Bell polynomials as
\[
\mu_{2n}\left(\Z\right)=B_{2n}\left(\kappa_{1}=0,\kappa_{2}=0,\kappa_{3}=0,\kappa_{4},\dots,\kappa_{2n}\right)
\]
where $\kappa_{2n},\thinspace\thinspace n\ge2,$ is the order $2n$
cumulant of $\Z$). For example, 
\[
\mu_{2}\left(\Z\right)=0,\thinspace\thinspace\mu_{4}\left(Z\right)=2\left(kk'\right)^{2}\left(\frac{z}{2}\right)^{4}, \mu_{6}\left(\Z\right)=8\left(kk'\right)^{2}\left(2k^{2}-1\right)\left(\frac{z}{2}\right)^{6}.
\]
This is since the moment  generating function $\varphi_{\theta_3}$ and cumulant generating function $\psi_{\theta_3}$ are related as 
$$\varphi_{\theta_{3}}\left(u\right) = \exp\left(\psi_{\theta_{3}}\left(u\right)\right),$$
and the complete Bell polynomials are defined as the exponential of an arbitrary formal power series:
$$\sum_{n=0}^\infty B_n(x_1,\ldots, x_n)\frac{t^n}{n!} :=  \exp \left(  \sum_{n=1}^\infty x_n \frac{t^n}{n!} \right).$$
Moreover, as will be shown in Theorem \ref{cumulants P} later on, the sequence of convolution polynomials $P_n\left( k \right)$ defined in Definition \ref{Pdef} is such that
\[
\kappa_{2n}=\left(-1\right)^{n-1}\left(\frac{z}{2}\right)^{2n}P_{2n-2}\left(k\right),\thinspace\thinspace n\ge2.
\]
We now appeal to the explicit formula (which easily follows from the generating function)
$$B_n(x_1,\ldots,x_k) = \sum_{k=1}^n \sum_{  \substack{ j_1, \ldots, j_{n-k+1}  \geq 0 \\ j_1 + j_2 +\cdots +j_{n-k+1} = k \\  j_1+2j_2+\cdots + (n-k+1)j_{n-k+1} = n }  }  \frac{n!}{j_1!j_2!\cdots j_{n-k+1}!}  \prod_{i=1}^{n-k+1} \left( \frac{x_i}{i!}\right)^{j_i}. $$
We observe that the complete Bell polynomial $B_{2n}$ is homogeneous of degree
$2n,$ and deduce that each monomial term in $B_{2n}$ will contribute a factor of $z^{2n}$, times a polynomial in $k$ whose degree is bounded by $2n$. Hence, for 
$n\ge2,$ and by the definition of $R_{2n}(k)$, we have
\[
\mu_{2n}\left(\Z\right)=\left(\frac{z}{2}\right)^{2n}R_{2n}\left(k\right).
\]
Moreover, since the polynomials $P_{2n-2}$ and the complete Bell
polynomials have integer coefficients, the corresponding polynomials
$R_{2n}$ have also integer coefficients.

We deduce the identity
\[
\sum_{p=-\infty}^{+\infty}q^{p^{2}}H_{2n}\left(\frac{p}{\sigma\sqrt{2}}\right)=\frac{2^{n}}{\sigma^{2n}}\theta_{3}\left(q\right)\left(\frac{z}{2}\right)^{2n}R_{2n}\left(k\right)
\]
which is equivalent to \eqref{eq:extProp10}.
\end{proof}
\begin{cor}
\label{corollary}
Romik's sequence $d\left(n\right)=1,1,-1,51\dots$ is related
to the polynomials $R_{n}\left(k\right)$ by
\[
d\left(n\right)=2^{n}R_{4n}\left(\frac{1}{\sqrt{2}}\right),\thinspace\thinspace n\ge1.
\]
\end{cor}
\begin{proof}

In the standard case $k=k'=\frac{1}{\sqrt{2}}$ and $ q = \exp \left( -\pi \frac{K'(k)}{K(k)}\right) = e^{-\pi}$, identity (\ref{eq:extProp10})
reduces to Romik's identity \cite[Proposition 10]{Romik}
\begin{equation}\label{rom:prop10}
\frac{1}{\theta_{3}\left(e^{- \pi}\right)}\sum_{p=-\infty}^{+\infty}e^{-\pi p^{2}}H_{2n}\left(\sqrt{2\pi}p\right)=\begin{cases}
2^{2n}\Phi^{\frac{n}{2}}d\left(\frac{n}{2}\right) & n\equiv0\mod2,\\
0 & n\equiv1\mod2,
\end{cases}
\end{equation}
with $\Phi := \frac{\Gamma^8(\frac14)}{128 \pi^4}$, since we will show that the left-hand sides of \eqref{eq:extProp10} and \eqref{rom:prop10} coincide. The corollary will follow from then equating the right-hand sides.

Substituting $k=k'=\frac{1}{\sqrt{2}}, q = e^{-\pi}$, we have
\[
4^{n}\Phi^{\frac{n}{2}}=2^{\frac{n}{2}}\frac{\Gamma^{4n}\left(\frac{1}{4}\right)}{\left(4\pi^{2}\right)^{n}}
\]
and 
\[
\frac{E\left(\frac{1}{\sqrt{2}}\right)}{K\left(\frac{1}{\sqrt{2}}\right)}-\frac{1}{2}=\frac{4\pi^{2}}{\Gamma^{4}\left(\frac{1}{4}\right)},
\]
so that
\[
\frac{8^{n}}{\left(\frac{E\left(\frac{1}{\sqrt{2}}\right)}{K\left(\frac{1}{\sqrt{2}}\right)}-\frac{1}{2}\right)^{n}}
=8^{n}\frac{\Gamma^{4n}\left(\frac{1}{4}\right)}{\left(4\pi^{2}\right)^{n}}.
\]
Then $\sigma^{2}=\frac{1}{4\pi}$ so that $\frac{1}{\sigma\sqrt{2}}
=\sqrt{2\pi}$ and the left-hand side of \eqref{eq:extProp10} becomes
$$\frac{1}{\theta_{3}\left(e^{- \pi}\right)}\sum_{p=-\infty}^{+\infty}e^{-\pi p^{2}}H_{2n}\left(\sqrt{2\pi}p\right).$$
Therefore \eqref{eq:extProp10} has correctly reduced to \eqref{rom:prop10}, and equating the right-hand sides of both identities yields the corollary.
\end{proof}
Our second main result is an explicit formula for the moments of a
discrete normal random variable with parameter $k$, as a finite sum.
It is an equivalent rephrasing of Theorem \ref{thm:sequence d(n)}.
\begin{thm}
\label{thm:moment theta3}With $q=e^{-\pi\frac{K'\left(k\right)}{K\left(k\right)}}$, $z=\theta_{3}^{2}\left(q\right)=\frac{2}{\pi}K\left(k\right)$,
and $\sigma^{2}$ given by (\ref{eq:cumulants}), the moments of a discrete
normal random variable can be computed using the sequence of polynomials
$\left\{ R_{2n}\right\} $ according to the formula
\begin{align*}
\frac{1}{\theta_{3}\left(q\right)}\sum_{p=-\infty}^{+\infty}p^{2n}q^{p^{2}} & =\sum_{j=0}^{n}\binom{2n}{2j}\frac{\left(2n-2j\right)!}{\left(n-j\right)!}\left(\frac{z}{2}\right)^{2j}R_{2j}\left(k\right)\left(-\frac{\sigma^{2}}{2}\right)^{n-j}.
\end{align*}
The standard case $k=\frac{1}{\sqrt{2}}$ reduces to Romik's identity
\cite[Proposition 9]{Romik}
\begin{align*}
\sum_{p=-\infty}^{+\infty}p^{2n}e^{-\pi p^{2}} & =\frac{\theta_{3}\left(e^{- \pi}\right)}{\left(4\pi\right)^{n}}\sum_{j=0}^{\left\lfloor \frac{n}{2}\right\rfloor }\frac{\left(2n\right)!}{2^{n-2j}\left(4j\right)!\left(n-2j\right)!}d\left(j\right)\Omega^{j}.
\end{align*}
\end{thm}

\begin{proof}
The proof is obtained by remarking that 
\[
\Z- i  \N_{\sigma^{2}}=\X_{\theta_{3}}
\]
and taking the $2n$ order moment. This is
\[
\mathbb{E}\X_{\theta_{3}}^{2n}=\frac{1}{\theta_{3}\left(q\right)}\sum_{p=-\infty}^{+\infty}p^{2n}q^{p^2}
\]
on one side, and 
\[
\mathbb{E}\left(\Z- i  \N_{\sigma^{2}}\right)^{2n}=\sum_{j=0}^{n}\binom{2n}{2j}\mathbb{E}\Z^{2j}\left(-1\right)^{n-j}\sigma^{2n-2j}\mathbb{E}\N_{1}^{2n-2j}
\]
on the other. We apply the standard relation
\[
\mathbb{E}\N_{1}^{2n}=\frac{1}{2^{n}}\frac{\left(2n\right)!}{n!},
\]
so that
\[
\sum_{p=-\infty}^{+\infty}p^{2n}
q^{p^2}=\theta_{3}\left(q\right)\sum_{j=0}^{n}\binom{2n}{2j}\mathbb{E}\Z^{2j}\left(-\frac{1}{2}\right)^{n-j}\sigma^{2n-2j}\frac{\left(2n-2j\right)!}{\left(n-j\right)!}.
\]
Substituting
\[
\mu_{2n}\left(\Z\right)=\left(\frac{z}{2}\right)^{2n}R_{2n}\left(k\right),
\]
we deduce the result.
\end{proof}

\section{Further results: Properties of The cumulants}

We first study the cumulants of the $\theta_{3}$ random variable.
A careful characterization of their properties will allow us to derive
some results about the moment generating function of the $\theta_{3}$ random variable
itself. Note that in order to prove our main results, Theorems \ref{thm:sequence d(n)} and \ref{thm:moment theta3}, we only require Theorem \ref{cumulants P} from this section. The rest of this section, however, consists of expansions of the cumulants that we believe will be useful in the future study of the $\theta_3$ distributed random variable.

\subsection{\label{subsec:Schett}Cumulants as Schett polynomials}

\begin{thm}\label{cumulants P}
\label{thm:cumulants as Schett}With $z=\theta_{3}^{2}\left(q\right)=\frac{2}{\pi}K\left(k\right),$
and $P_{2n}(k)$ given in Definition \ref{Pdef}, the cumulants of $\X_{\theta_{3}}$ are expressed as
\begin{equation}
\kappa_{2n}=\left(-1\right)^{n-1}\left(\frac{z}{2}\right)^{2n}P_{2n-2}\left(k\right),\thinspace\thinspace n\ge2.\label{eq:kappa P}
\end{equation}
\end{thm}

\begin{proof}
It is shown in  \cite{Dumont 2} that a moment generating function for the Schett polynomials
$X_{n}\left(0,a,b\right)$ is
\[
\frac{1}{ i  a}\sn\left( i  au,\frac{b}{a}\right)=\frac{1}{ab}\sum_{n=0}^{\infty}X_{n}\left(0,a,b\right)\frac{u^{n}}{n!}
\]
where $\sn\left( u,a \right)$ is the Jacobi elliptic function.
Since this is an odd function of $u,$ we rewrite this equivalently
as
\[
\frac{1}{ i  a}\sn\left( i  au,\frac{b}{a}\right)=\frac{1}{ab}\sum_{n=0}^{\infty}X_{2n+1}\left(0,a,b\right)\frac{u^{2n+1}}{\left(2n+1\right)!},
\]
so that, choosing $a=k$ and $b= i  k',$
\[
\frac{1}{ i  k}\sn\left( i  ku,\frac{ i  k'}{k}\right)=\frac{1}{ i  kk'}\sum_{n=0}^{\infty}X_{2n+1}\left(0,k, i  k'\right)\frac{u^{2n+1}}{\left(2n+1\right)!}.
\]
We deduce
\[
-\frac{1}{k^{2}}\sn^{2}\left( i  ku,\frac{ i  k'}{k}\right)=-\frac{1}{\left(kk'\right)^{2}}\sum_{p=0}^{\infty}\frac{u^{2p+2}}{\left(2p+2\right)!}\sum_{n=0}^{p}\binom{2p+2}{2n+1}X_{2n+1}\left(0,k, i  k'\right)X_{2p-2n+1}\left(0,k, i  k'\right).
\]
Recalling the definition of $P_{2p}(k)$, we deduce the generating function
\[
\frac{1}{k^{2}}\sn^{2}\left( i  ku,\frac{ i  k'}{k}\right)=\frac{1}{\left(kk'\right)^{2}}\sum_{p=0}^{\infty}\frac{u^{2p+2}}{\left(2p+2\right)!}P_{2p+2}\left(k\right)
\]
for the sequence $P_{2p}(k).$
Comparing to the generating function (see section \ref{subsec:The-cumulant-generating}
below)
\begin{equation}
\label{cumulants sn2}
\frac{4}{\left(kk'\right)^{2}}\sum_{m=1}^{\infty}\frac{\left(-1\right)^{m}2^{2m}}{z^{2m+2}}\kappa_{2m+2}\frac{u^{2m}}{\left(2m\right)!}=\frac{1}{k^{2}}\sn^{2}\left( i  uk, i \frac{k'}{k}\right)
\end{equation}
produces
\[
\frac{4}{\left(kk'\right)^{2}}\frac{\left(-1\right)^{m}2^{2m}}{z^{2m+2}}\kappa_{2m+2}=\frac{1}{\left(kk'\right)^{2}}P_{2m}\left(k\right),
\]
or
\[
\kappa_{2n}=\left(-1\right)^{n-1}\left(\frac{z}{2}\right)^{2n}P_{2n-2}\left(k\right).
\]
\end{proof}



Note that this expression is not valid for $n=1$; instead, we have the identity $\kappa_2 = \sigma^2$ with $\sigma^2$ defined  in \eqref{eq:cumulants}.

Theorem \ref{thm:cumulants as Schett} provides a refinement of a result of Shaun Cooper and Heung Yeung Lam \cite[Thm. 0.3]{Cooper}, 
who express the cumulants under the form
\[
\kappa_{2n}=z^{2n}\left(kk'\right)^{2}p_{n-2}\left(k^{2}\right),\thinspace\thinspace n\ge2,
\]
for some polynomial $p_{n-2}$ of degree $n-2$ with rational coefficients. 
The previous result shows the link between these $p_n$ polynomials, 
which were not further characterized, and the Schett polynomials. 

Cooper  and Lam in fact considered \textit{sixteen} different families of Eisenstein series and wrote each as a prefactor times a polynomial of restricted degree, which they did not characterize further. We have provided the explicit polynomial in one of these cases; an interesting result would be the identification of the other fifteen.

\subsection{Cumulants as Lambert series}
\begin{thm}\label{lambert}
The even cumulants of $\X_{\theta_{3}}$, with $c = \frac{K'(k)}{K(k)}$ and $q=e^{-c \pi}$, can be expressed as the Lambert
series
\begin{equation}
\kappa_{2n}=\sum_{k=1}^{\infty}\frac{\left(-1\right)^{k-1}k^{2n-1}}{\sinh\left(ck\pi\right)},\thinspace\thinspace n\ge1,\label{eq:cumulants theta3}
\end{equation}
while 
\[
\kappa_{2n+1}=0,\thinspace\thinspace n\ge 0.
\]
\end{thm}

\begin{proof}
Since
\[
\theta_{3}\left(z;q\right)=\sum_{n=-\infty}^{\infty}q^{n^{2}}e^{ i 2nz}
\]
the moment generating function for $\X_{\theta_{3}}$ is
\[
\mathbb{E}e^{z\X_{\theta_{3}}}=\frac{1}{\theta_{3}\left(0,q\right)}\sum_{n=-\infty}^{\infty}q^{n^{2}}e^{nz}=\frac{\theta_{3}\left(\frac{z}{2 i },q\right)}{\theta_{3}\left(0,q\right)}.
\]
Using the infinite product representation (\ref{eq:theta3 infinite product})
for the $\theta_{3}$ function gives
\begin{align*}
\frac{\theta_{3}\left(\frac{z}{2 i },q\right)}{\theta_{3}\left(0,q\right)} & =\prod_{p=0}^\infty\frac{\left(1+e^{z}q^{2p+1}\right)}{\left(1+q^{2p+1}\right)}\frac{\left(1+e^{-z}q^{2p+1}\right)}{\left(1+q^{2p+1}\right)}.
\end{align*}
Defining the function
\[
f\left(z\right):=\sum_{p=0}^{\infty}\log\left(1+e^{z}q^{2p+1}\right),
\]
the cumulants of $\X_{\theta_{3}}$ can be computed as the Taylor coefficients
of 
\[
\log\frac{\theta_{3}\left(\frac{z}{2 i },q\right)}{\theta_{3}\left(0,q\right)}=f\left(z\right)+f\left(-z\right)-2f\left(0\right).
\]
Since 
\begin{align*}
f\left(z\right)=\sum_{p=0}^{\infty}\log\left(1+e^{z}q^{2p+1}\right) & =\sum_{p=0}^{\infty}\sum_{k=1}^{\infty}\frac{\left(-1\right)^{k+1}}{k}e^{kz}q^{k\left(2p+1\right)}\\
 & =\sum_{p=0}^{\infty}\sum_{k=1}^{\infty}\frac{\left(-1\right)^{k+1}}{k}\sum_{n=0}^{\infty}\frac{\left(kz\right)^{n}}{n!}q^{k\left(2p+1\right)}.\\
\end{align*}
Interchanging the inner and outer sum, we deduce 
\[
f\left(z\right)=\sum_{n=0}^{\infty}\frac{z^{n}}{n!}\sum_{k=1}^{\infty}\left(-1\right)^{k+1}k^{n-1}\frac{q^{k}}{1-q^{2k}}
\]
and 
\begin{align*}
f\left(0\right) & =\sum_{k=1}^{\infty}\frac{\left(-1\right)^{k+1}}{k}\frac{q^{k}}{1-q^{2k}}.
\end{align*}
The cumulant generating function is then
\begin{align*}
\log\frac{\theta_{3}\left(\frac{z}{2 i };q\right)}{\theta_{3}\left(0;q\right)} & =2\sum_{n=1}^{\infty}\frac{z^{2n}}{2n!}\sum_{k=1}^{\infty}\left(-1\right)^{k+1}k^{2n-1}\frac{q^{k}}{1-q^{2k}},
\end{align*}
and the cumulants are identified as
\[
\kappa_{2n+1}=0,\thinspace\thinspace n\ge0,
\]
\[
\kappa_{2n}=2\sum_{k=1}^{\infty}\frac{\left(-1\right)^{k-1}k^{2n-1}}{q^{-k}-q^{k}},\thinspace\thinspace n\ge1.
\]
With $q=e^{-c\pi},$ this is the desired result.
\end{proof}

\subsection{\label{subsec:The-cumulant-generating}Cumulant generating function}

The cumulant generating function for the discrete normal random variable
can be expressed in terms of the Jacobi elliptic function $\sd\left(u,k\right)$
using a result by Milne \cite[Eq.2.43]{Milne}:
\begin{equation}
\sd^{2}\left(u,k\right)=-\frac{1}{k^{2}}+\frac{1}{\left(kk'\right)^{2}}\frac{E\left(k\right)}{K\left(k\right)}-\frac{8}{\left(kk'\right)^{2}}\sum_{m=0}^{\infty}\frac{\left(-1\right)^{m}2^{2m}}{z^{2m+2}}\left(\sum_{r=1}^{\infty}\frac{\left(-1\right)^{r-1}r^{2m+1}q^{r}}{1-q^{2r}}\right)\frac{u^{2m}}{\left(2m\right)!}\label{eq:sd2}
\end{equation}
with the notation
\[
z=\frac{2}{\pi}K\left(k\right)=\theta_{3}^{2}\left(q\right),\thinspace\thinspace q=e^{-\pi\frac{K\left(k'\right)}{K\left(k\right)}}.
\]
 Using Theorem \ref{lambert}, we recognize the inner sum over $r$ as precisely the Lambert series expansion of these cumulants. A consequence of this representation is as follows: since 
\[
\sd^{2}\left(u,k\right)=u^{2}+O\left(u^{4}\right),
\]
we deduce 
\[
-\frac{1}{k^{2}}+\frac{1}{\left(kk'\right)^{2}}\frac{E\left(k\right)}{K\left(k\right)}-\frac{8}{\left(kk'\right)^{2}}\frac{1}{z^{2}}\kappa_{2}=0
\]
which provides the value of the variance $\sigma^{2}=\kappa_{2}$
as expressed by (\ref{eq:cumulants}).

\subsection{Cumulants as Eisenstein series}

We now use a result by Ling \cite{Ling} to provide an alternate expression
for the cumulants as Eisenstein series.
\begin{thm}
With $c=\frac{K'\left( k \right)}{K\left( k \right)}$, the even cumulants of $\X_{\theta_{3}}$  are, for $n\ge1,$
\begin{align*}
\kappa_{2n} & =\frac{2\left(-1\right)^{n+1}\left(2n-1\right)!}{\pi^{2n}}\mathop{\sum_{n_{1}\ge1}}_{n_{2}\in\mathbb{Z}}\frac{1}{\left(2n_{1}-1+ i  c\left(2n_{2}-1\right)\right)^{2n}}\\
 & =\frac{\left(-1\right)^{n+1}\left(2n-1\right)!}{\pi^{2n}}\mathop{\sum_{n_{1},n_{2}\in\mathbb{Z}}}\frac{1}{\left(2n_{1}-1+ i  c\left(2n_{2}-1\right)\right)^{2n}}.
\end{align*}
\end{thm}

\begin{proof}
The first expression is \cite[Eq.(14)]{Ling}. It is obtained using
the Mittag-Leffler expansion
\[
\tanh\left(\pi x\right)=\frac{8x}{\pi}\sum_{m=1}^{\infty}\frac{1}{\left(2m-1\right)^{2}+4x^{2}}
\]
together with the partial fraction decomposition
\[
\sum_{p=1}^{\infty}\left(-1\right)^{p}e^{-2\pi px}=-\frac{1}{2}+\frac{ i }{\pi}\sum_{m=1}^{\infty}\left(\frac{1}{2m-1+2 i  x}-\frac{1}{2m-1-2 i  x}\right).
\]
The second identity is deduced from the first by symmetry.
\end{proof}

\subsection{Cumulants and Combinatorics}

In the concluding Open Problems section of his recent article \cite{Romik},
Dan Romik asked for a combinatorial interpretation of the sequence $d\left(n\right).$
We were not able to find such an interpretation, but we can provide
one for the sequence of cumulants as follows.


Consider a permutation $\sigma\in\mathfrak{S}_{n}$ and denote $\sigma^{-1}$
its inverse. A cycle peak of $\sigma$ is an integer $k$ such that
$2\le k\le m$ and 
\[
\sigma\left(k\right)\ne k,\thinspace\thinspace\sigma\left(k\right)<k\thinspace\thinspace\text{and}\thinspace\thinspace\sigma^{-1}\left(k\right)<k.
\]
Dumont \cite{Dumont 2} provides the example
\[
\sigma=\left(134\right)\left(2\right)\left(56\right),
\]
for which the cycle peaks are 4 and 6. Denote $P_{n,i,j}$ the number
of permutations $\sigma\in\mathfrak{S}_{n}$ that have $i$ odd cycle
peaks and $j$ even cycle peaks. The link with elliptic functions
is given by the following result, where  we notice that the function $\sn\left(u;a,b\right)$ as introduced
by Dumont \cite{Dumont} is related to the classical Jacobi $sn\left(u;k\right)$
function by
\[
\sn\left(u;a,b\right)=\frac{1}{ i  a}\sn\left( i  au;\frac{b}{a}\right).
\]

\begin{thm}\cite[Corollaire 8.7]{Dumont}
The coefficient of $a^{2j}b^{2n-2j}\frac{u^{2n}}{\left(2n\right)!}$
in the Taylor expansion of the function $\frac{1}{2}\sn^{2}\left(u,a,b\right)$
is equal to the number $P_{2n,1,j}.$
\end{thm}

Notice that the first terms in the Taylor expansion of $\frac{1}{2}\sn^{2}\left(u,a,b\right)$
are
\[
\frac{u^{2}}{2!}+4(a^{2}+b^{2})\frac{u^{4}}{4!}+8(2a^{4}+13a^{2}b^{2}+2b^{4})\frac{u^{6}}{6!}+\dots
\]
so that the coefficient of $\frac{u^{2n}}{\left(2n\right)!}$ is an
homogeneous polynomial in $\left(a,b\right)$ of degree $n-2$, and
not $n$ as suggested by Dumont's result. Tracing through Dumont's proof reveals that the correct statement is in fact:
\begin{thm}
The coefficient of $a^{2j-2}b^{2n-2j}\frac{u^{2n}}{\left(2n\right)!}$
in the Taylor expansion of the function $\frac{1}{2}\sn^{2}\left(u,a,b\right)$
is equal to the number $P_{2n,1,j}.$
\end{thm}

For example, in the case $n=1,$ the only 2 permutations are $\left(1\right)\left(2\right)$
and $\left(2,1\right).$ The first has zero even cycle peaks and zero
odd cycle peaks, so that $P_{2,1,1}=0,$ while the second has one odd
cycle peak and zero even cycle peaks so that $P_{2,1,0}=1.$ The coefficient
of $a^{-2}b^{0}\frac{u^{2}}{2!}$ is $P_{2,1,1}=0$ and the coefficient
of $a^{0}b^{0}\frac{u^{2}}{2!}$ is $P_{2,1,0}=1.$ 

We deduce the following result.
\begin{thm}
\label{thm:thm11}With $z=\theta_{3}^{2}\left(q\right)=\frac{2}{\pi}K\left(k\right),$
the cumulant $\kappa_{2n+2}$ is equal to
\[
\kappa_{2n+2}=\frac{z^{2n+2}}{2^{2n+1}}\sum_{j=0}^{n}\left(-1\right)^{j-1}k^{2j+2}\left(k'\right)^{2n-2j+2}P_{2n,1,j}.
\]
\end{thm}

\begin{proof}
We relate the cumulant generating function $\sd\left( u,k \right)$ to the function
$\sn\left( u,k \right)$: this can be done by applying first the Jacobi real transformation
\cite[13.34]{Neville}
\[
\sd\left(u,k\right)=\frac{1}{k}\scellip\left(uk,\frac{1}{k}\right),
\]
and by noticing that the Jacobi elliptic functions $sc$ and $sn$ are related
by the Jacobi imaginary transformation \cite[13.25]{Neville} as
\[
\scellip\left(u,k\right)=- i  \sn\left( i  u,k'\right).
\]
We deduce
\[
\sd^{2}\left(u,k\right)=-\frac{1}{k^{2}}\sn^{2}\left( i  uk, i \frac{k'}{k}\right).
\]
Now Dumont's result is the generating function
\[
\sum_{n=1}^{\infty}\frac{u^{2n}}{\left(2n\right)!}\sum_{j=0}^{n}b^{2j}c^{2n-2j}P_{2n,1,j}=\frac{1}{2}\sn^{2}\left(u,b,c\right)=-\frac{1}{2b^{2}}\sn^{2}\left( i  bu,\frac{c}{b}\right).
\]
Choosing \textbf{$c= i  k',b=k$} produces
\[
\sum_{n=1}^{\infty}\sum_{j=0}^{n}\frac{\left(-1\right)^{n-j}u^{2n}}{\left(2n\right)!}k^{2j}\left(k'\right)^{2n-2j}P_{2n,1,j}=-\frac{1}{2k^{2}}\sn^{2}\left( i  ku,\frac{ i  k'}{k}\right),
\]
whereas the cumulant generating function is
\[
-\frac{4}{\left(kk'\right)^{2}}\sum_{m=0}^{\infty}\frac{\left(-1\right)^{m}2^{2m}}{z^{2m+2}}\kappa_{2m+2}\frac{u^{2m}}{\left(2m\right)!}=\sd^{2}\left(u,k\right)=-\frac{1}{k^{2}}\sn^{2}\left( i  uk, i \frac{k'}{k}\right).
\]
Identifying the coefficient of $\frac{u^{2m}}{\left(2m\right)!}$
in each expression produces
\[
-\frac{4}{\left(kk'\right)^{2}}\frac{\left(-1\right)^{m}2^{2m}}{z^{2m+2}}\kappa_{2m+2}=2\sum_{j=0}^{m}\left(-1\right)^{m-j}k^{2j}\left(k'\right)^{2m-2j}P_{2m,1,j}
\]
or
\[
\kappa_{2m+2}=\frac{z^{2m+2}}{2^{2m+1}}\sum_{j=0}^{m}\left(-1\right)^{j-1}k^{2j+2}\left(k'\right)^{2m-2j+2}P_{2m,1,j}.
\]
\end{proof}
\begin{cor}
In the standard case $k=k'=\frac{1}{\sqrt{2}},$ we have $K\left(\frac{1}{\sqrt{2}}\right)=\frac{\Gamma^{2}\left(\frac{1}{4}\right)}{4\sqrt{\pi}},$
$z=\frac{2}{\pi}K\left( \frac{1}{\sqrt{2}} \right)\frac{\Gamma^{2}\left(\frac{1}{4}\right)}{2\pi^{\frac{3}{2}}}$
and the expansion of cumulants simplifies to
\[
\kappa_{2n+2}=\frac{\Gamma^{4n+2}\left(\frac{1}{4}\right)}{2^{3n+4}\pi^{\frac{3}{2}}}\sum_{j=0}^{n}\left(-1\right)^{j-1}P_{2n,1,j}.
\]
\end{cor}

We include here a series of additional remarks about the standard case $k = \frac{1}{\sqrt{2}}$.
\begin{enumerate}
\item The sequence $\left\{ Q_{n}\right\} $ defined by 
\begin{equation}
Q_{2n}=\sum_{j=0}^{n-1}\left(-1\right)^{j-1}P_{2n-2,1,j}\label{eq:Q2n}
\end{equation}
counts the difference between the number of permutations of $\left[1,n \right]$ with one
odd cycle peak and an odd number of even cycle peaks and the number
of permutations with one odd cycle peak and an even number of even
cycle peaks. Moreover
\[
Q_{4}=2,\thinspace\thinspace Q_{6}=0,\thinspace\thinspace Q_{8}=-144,\thinspace\thinspace Q_{10}=0,\thinspace\thinspace Q_{12}=96768.
\]
\item The sequence $\left\{ \frac{1}{2}Q_{2n}\right\} $ appears as OEIS
A260779 and coincides with the sequence of Taylor coefficients of
the reciprocal of Weierstrass' $\wp$ function in the lemniscatic
case, for which the invariants are $g_2=4$ and $g_3=0$, and the periods are
$\omega_1=\frac{\Gamma^2\left( \frac{1}{4} \right)}{2\sqrt{2\pi}}$ and $\omega_2=\left( 1+i \right)\frac{\Gamma^2\left( \frac{1}{4} \right)}{2\sqrt{2\pi}}$. More precisely,
\[
\frac{1}{\wp\left(z;\omega_1,\omega_2 \right)}=2\frac{z^{2}}{2!}-144\frac{z^{6}}{6!}+96768\frac{z^{10}}{10!}+\dots
\]
This sequence was first studied by Hurwitz \cite{Hurwitz}. It is also proportional
by a factor $\left(-12\right)^{n}$ to the sequence OIES A144849 of
Taylor coefficients of the square of the sine lemniscate function
\[
\slellip\left(u\right)=\frac{1}{\sqrt{2}}\sd\left(u\sqrt{2},\frac{1}{\sqrt{2}}\right).
\]
\item The result of Thm \ref{thm:thm11} can be restated as follows: in
the standard case $k=\frac{1}{\sqrt{2}},$ the cumulants are, for
$n\ge2$ and with $z=\frac{\Gamma^{2}\left(\frac{1}{4}\right)}{2\pi^{\frac{3}{2}}},$
\[
\kappa_{2n}=\left(\frac{z}{2\sqrt{2}}\right)^{2n}Q_{2n}
\]
with $Q_{2n}\in\mathbb{Z}$ given by (\ref{eq:Q2n}). The first cases
are
\[
Q_{4}=2,\thinspace\thinspace\kappa_{4}=2\left(\frac{z}{2\sqrt{2}}\right)^{4}=\frac{1}{2^{9}}\frac{\Gamma^{8}\left(\frac{1}{4}\right)}{\pi^{6}},
\]
\[
Q_{6}=0,\thinspace\thinspace\kappa_{6}=0,
\]
\[
Q_{8}=-144,\thinspace\thinspace\kappa_{8}=-144\left(\frac{z}{2\sqrt{2}}\right)^{6}.
\]
\end{enumerate}

\section{Symmetries of the Moments and Cumulants}

The sequences of polynomials $\left\{ P_{2n}\left(k\right)\right\} $
and $\left\{ R_{2n}\left(k\right)\right\} $ and the sequences of
moments and cumulants associated with the discrete normal distribution
exhibit a natural symmetry with respect to the transformation $k \mapsto k'=\sqrt{1-k^2}$ of the elliptic modulus,
as expressed in the following theorem.
\begin{thm}
For $k'=\sqrt{1-k^{2}}$ and $n\ge1,$
\[
P_{2n}\left(k'\right)=\left(-1\right)^{n-1}P_{2n}\left(k\right),
\]
and
\[
R_{2n}\left(k'\right)=\left(-1\right)^{n}R_{2n}\left(k\right),
\]
so that 
\[
R_{4n}\left(k'\right)=R_{4n}\left(k\right).
\]
Moreover, the cumulants are related as
\[
\kappa_{2n}\left(k'\right)=\left( i \frac{K\left(k'\right)}{K\left(k\right)}\right)^{2n}\kappa_{2n}\left(k\right).
\]
For $n\ge2,$ the moments $\mu_{2n}\left(k\right)$ and $\mu_{2n}\left(k'\right)$
are related as 
\[
\mu_{2n}\left(k'\right)=\sum_{j=0}^{n}\binom{2n}{2j}\frac{\left(2n-2j\right)!}{\left(n-j\right)!}\left( i \frac{K\left(k'\right)}{K\left(k\right)}\right)^{2j}\left(\frac{\delta}{\sqrt{2}}\right)^{2n-2j}\mu_{2j}\left(k\right),
\]
with $\delta^{2}=\sigma^{2}\left(k'\right)-\sigma^{2}\left(k\right)$. These variances are related as
\begin{equation}
\frac{\sigma^{2}\left(k\right)}{K^{2}\left(k\right)}+\frac{\sigma^{2}\left(k'\right)}{K^{2}\left(k'\right)}=\frac{1}{2\pi}\frac{1}{K\left(k\right)K\left(k'\right)}.\label{eq:variances}
\end{equation}
\end{thm}

\begin{proof}
The invariance of the polynomials $P_{2n}$ and $R_{2n}$ is a consequence of the invariance of the Schett polynomials
$X_{2n+1}\left( 0, k', i  k \right) = \left( -1 \right)^n
X_{2n+1}\left( 0, k, i  k' \right).
$ We skip the details.

From (\ref{eq:kappa P}),
\begin{align*}
\kappa_{2n}\left(k'\right) & =\left(-1\right)^{n-1}\left(\frac{z'}{2}\right)^{2n}P_{2n-2}\left(k'\right)=\left(-1\right)^{n-1}\left(\frac{z'}{2}\right)^{2n}\left(-1\right)^{n}P_{2n-2}\left(k\right),
\end{align*}
with
\[
\frac{z'}{2}=\frac{K\left(k'\right)}{\pi}=\frac{K\left(k'\right)}{K\left(k\right)}\frac{z}{2},
\]
so that
\[
\kappa_{2n}\left(k'\right)=-\left(\frac{K\left(k'\right)}{K\left(k\right)}\right)^{2n}\left(\frac{z}{2}\right)^{2n}P_{2n-2}\left(k\right)=\left(-1\right)^{n}\left(\frac{K\left(k'\right)}{K\left(k\right)}\right)^{2n}\kappa_{2n}\left(k\right).
\]
Using Legendre's identity
\[
K\left(k\right)E\left(k'\right)+K\left(k'\right)E\left(k\right)-K\left(k\right)K\left(k'\right)=\frac{\pi}{2}
\]
and the expression of the variance $\sigma^{2}\left(k\right)$ in
(\ref{eq:cumulants}), we deduce (\ref{eq:variances}).
\end{proof}
\begin{rem}
An equivalent statement of the previous result is as follows: for $\X_{k}$ and $\X_{k'}$ two discrete normal random variables with
respective elliptic moduli $k$ and $k'$, and two standard Gaussian
random variables $\N$ and $\N',$ all four random variables mutually independent, the  two random variables 
\[
\X_{k'}+\sigma_{k'}^{2}\N'\thinspace\thinspace\text{and}\thinspace\thinspace\left( i \frac{K\left(k'\right)}{K\left(k\right)}\right)\X_{k}+\sigma_{k}^{2}\N
\]
have the same moments and cumulants.
\end{rem}

\section{A numerical approach}

A numerical toolbox \cite{Zeilberger} written by D. Zeilberger allows us to compute the
sequence $d\left(n\right)$ (and many other quantities related to
Romik's paper) under the Maple environment. In this toolbox, the first 200 values of
$d\left(n\right)$ are precomputed while the higher-order ones are
computed using a recurrence formula derived by D. Romik \cite[Thm 7]{Romik} that
requires the evaluation of the Taylor coefficients of the two functions
\[
U\left(t\right)=\frac{_{2}F_{1}\left(\frac{3}{4},\frac{3}{4};\frac{3}{2};4t\right)}{_{2}F_{1}\left(\frac{1}{4},\frac{1}{4};\frac{1}{2};4t\right)},\thinspace\thinspace V\left(t\right)=\sqrt{_{2}F_{1}\left(\frac{1}{4},\frac{1}{4};\frac{1}{2};4t\right)}.
\]
We propose here another method based on the prior computation of the
cumulants using a quadratic recurrence, and on a linear recurrence between the cumulants and the
moments. More precisely, the computation of the cumulants through the quadratic recurrence of Theorem \ref{thm13} below is followed by an application of the general moments-cumulants recurrence \cite{Rota}
\begin{equation}
\mu_{n}=\kappa_{n}+\sum_{m=1}^{n-1}\binom{n-1}{m-1}\kappa_{m}\mu_{n-m}, \thinspace\thinspace n \geq 2.\label{eq:moments cumulants}
\end{equation}
\begin{thm}\label{thm13}
The cumulants $\kappa_{2n}$ of the discrete distribution with parameter
$k$ satisfy the recurrence
\begin{align*}
\kappa_{2n+2} & =\left(1-2k^2\right)z^{2}\kappa_{2n}-6\sum_{\nu=1}^{n-2}\binom{2n-2}{2\nu}\kappa_{2\nu+2}\kappa_{2n-2\nu}.
\end{align*}
\end{thm}
\begin{proof}
We use the result from \cite[Theorem 4]{Wrigge}: let $\lambda_{n}\left(k\right)$ denote the Taylor coefficients of the function $\sn^{2}\left(u,k\right)$
so that
\[
\sn^{2}\left(u,k\right)=\sum_{n=1}^{\infty}\frac{\lambda_{2n}\left(k\right)}{\left(2n\right)!}u^{2n}.
\]
We know from \eqref{cumulants sn2} that the cumulants are such that 

\[
\frac{4}{\left(kk'\right)^{2}}\sum_{n=1}^{\infty}\frac{\left(2 i \right)^{2n}}{z^{2n+2}}\kappa_{2n+2}\frac{u^{2n}}{\left(2n\right)!}=\frac{1}{k^{2}}\sum_{n=1}^{\infty}\frac{\lambda_{2n}\left( i \frac{k'}{k}\right)}{\left(2n\right)!}\left( i  k\right)^{2n}u^{2n}.
\]
We deduce
\[
\lambda_{2n}\left( i \frac{k'}{k}\right)=\frac{k^{2}}{\left( i  k\right)^{2n}}\frac{4}{\left(kk'\right)^{2}}\frac{\left(2 i \right)^{2n}}{z^{2n+2}}\kappa_{2n+2}=\frac{1}{k^{2n}}\frac{1}{\left(k'\right)^{2}}\frac{2^{2n+2}}{z^{2n+2}}\kappa_{2n+2}
\]
The coefficients $\lambda_{2n}\left(k\right)$ satisfy the recurrence
\cite[Theorem 4]{Wrigge}
\[
\lambda_{2n+2}\left(k\right)=-4\left(1+k^{2}\right)\lambda_{2n}\left(k\right)+6k^{2}\sum_{\nu=1}^{n-1}\binom{2n}{2\nu}\lambda_{2\nu}\left(k\right)\lambda_{2n-2\nu}\left(k\right),
\]
so that the coefficients $\lambda_{2n}\left( i \frac{k'}{k}\right)$
satisfy the recurrence
\[
\lambda_{2n+2}\left( i \frac{k'}{k}\right)=-4\left(1-\left(\frac{k'}{k}\right)^{2}\right)\lambda_{2n}\left( i \frac{k'}{k}\right)-6\left(\frac{k'}{k}\right)^{2}\sum_{\nu=1}^{n-1}\binom{2n}{2\nu}\lambda_{2\nu}\left( i \frac{k'}{k}\right)\lambda_{2n-2\nu}\left( i \frac{k'}{k}\right).
\]
Substituting yields
\begin{align*}
\frac{1}{k^{2n+2}}\frac{1}{\left(k'\right)^{2}}\frac{2^{2n+4}}{z^{2n+4}}\kappa_{2n+4}  
  =-4&\left(1-\left(\frac{k'}{k}\right)^{2}\right)\frac{1}{k^{2n}}\frac{1}{\left(k'\right)^{2}}\frac{2^{2n+2}}{z^{2n+2}}\kappa_{2n+2}\\
 & -6\frac{1}{k^{2n+2}\left(k'\right)^{2}}\frac{2^{2n+4}}{z^{2n+4}}\sum_{\nu=1}^{n-1}\binom{2n}{2\nu}\kappa_{2\nu+2}\kappa_{2n-2\nu+2},
\end{align*}
or
\begin{equation*}
\kappa_{2n+2} =-\left(k^{2}-\left(k'\right)^{2}\right)z^{2}\kappa_{2n}-6\sum_{\nu=1}^{n-2}\binom{2n-2}{2\nu}\kappa_{2\nu+2}\kappa_{2n-2\nu}.
\end{equation*}
Noticing that $k^{2}-\left(k'\right)^{2} = 2k^2-1$ completes the proof.
\end{proof}
For the computation of the moments, the mixed recurrence (\ref{eq:moments cumulants})
between moments and cumulants can also be replaced by a direct formula
for the moments in terms of the cumulants, at the price of the computation
of a determinant (see \cite{Rota}).
\begin{thm}
A determinant representation of the moments of the standard discrete
normal distributions is
\[
\mu_{n}=\left(-1\right)^{n-1}\left(n-1\right)!\det\left[\begin{array}{ccccccc}
\frac{\kappa_{2}}{1!} & 1 & 0 & 0 & 0 & \cdots & 0\\
\frac{\kappa_{3}}{2!} & \frac{\kappa_{1}}{2.0!} & 1 & 0 & 0\\
\frac{\kappa_{4}}{3!} & \frac{\kappa_{2}}{2.1!} & \frac{\kappa_{1}}{3.0!} & 1 & 0\\
\frac{\kappa_{5}}{4!} & \frac{\kappa_{3}}{2.2!} & \frac{\kappa_{2}}{3.1!} & \frac{\kappa_{1}}{4.0!} & 1\\
\vdots & \vdots & \vdots & \vdots &  & \ddots & 0\\
 &  &  &  &  &  & 1\\
\frac{\kappa_{n}}{\left(n-1\right)!} & \frac{\kappa_{n-2}}{2\left(n-3\right)!} & \frac{\kappa_{n-3}}{3\left(n-4\right)!} & \frac{\kappa_{n-4}}{4\left(n-5\right)!} &  &  & \frac{\kappa_{1}}{\left(n-1\right).0!}
\end{array}\right],\thinspace\thinspace n\ge2.
\]
\end{thm}

\begin{proof}
The moments are related to the cumulants via \cite{Rota}
\[
\frac{\kappa_{n}}{\left(n-1\right)!}=\frac{\mu_{n}}{\left(n-1\right)!}-\sum_{m=0}^{n-2}\frac{\kappa_{m+1}}{m!}\frac{\mu_{n-m-1}}{\left(n-m-2\right)!}\frac{1}{n-m-1}.
\]
With the notation
\begin{equation}
\tilde{\kappa}_{n}=\frac{\kappa_{n}}{\left(n-1\right)!},\thinspace\thinspace\tilde{\mu}_{n}=\frac{\mu_{n}}{\left(n-1\right)!},\label{eq:kappatilde}
\end{equation}
this is expressed as the linear system
\[
\left[\begin{array}{c}
\tilde{\kappa}_{2}\\
\tilde{\kappa}_{3}\\
\tilde{\kappa}_{4}\\
\tilde{\kappa}_{5}\\
\vdots
\end{array}\right]=\left[\begin{array}{ccccc}
1 & 0 & 0 & 0 & \cdots\\
-\tilde{\frac{1}{2}\kappa_{1}} & 1 & 0 & 0\\
-\frac{1}{2}\tilde{\kappa}_{2} & -\frac{1}{3}\tilde{\kappa}_{1} & 1 & 0\\
-\frac{1}{2}\tilde{\kappa}_{3} & -\frac{1}{3}\tilde{\kappa}_{2} & -\frac{1}{4}\tilde{\kappa}_{1} & 1\\
\vdots & \vdots & \vdots
\end{array}\right]\left[\begin{array}{c}
\tilde{\mu}_{2}\\
\tilde{\mu}_{3}\\
\tilde{\mu}_{4}\\
\tilde{\mu}_{5}\\
\vdots
\end{array}\right].
\]
Using Cramer's formula to invert this linear system, we deduce
\begin{align*}
\tilde{\mu}_{n} & =\det\left[\begin{array}{ccccccc}
\tilde{\kappa}_{2} & 1 & 0 & 0 & 0 & \cdots\\
\tilde{\kappa}_{3} & -\frac{\tilde{\kappa}_{1}}{2} & 1 & 0 & 0\\
\tilde{\kappa}_{4} & -\frac{\tilde{\kappa}_{2}}{2} & -\frac{\tilde{\kappa}_{1}}{3} & 1 & 0\\
\tilde{\kappa}_{5} & -\frac{\tilde{\kappa}_{3}}{2} & -\frac{\tilde{\kappa}_{2}}{3} & -\frac{\tilde{\kappa}_{1}}{4} & 1\\
\vdots & \vdots & \vdots & \vdots &  & \ddots\\
 &  &  &  &  &  & 1\\
\tilde{\kappa}_{n} & -\frac{\tilde{\kappa}_{n-2}}{2} & -\frac{\tilde{\kappa}_{n-3}}{3} & -\frac{\tilde{\kappa}_{n-4}}{4} &  &  & -\frac{\tilde{\kappa}_{1}}{n-1}
\end{array}\right]
  =\left(-1\right)^{n-1}\det\left[\begin{array}{ccccccc}
\tilde{\kappa}_{2} & 1 & 0 & 0 & 0 & \cdots\\
\tilde{\kappa}_{3} & \frac{\tilde{\kappa}_{1}}{2} & 1 & 0 & 0\\
\tilde{\kappa}_{4} & \frac{\tilde{\kappa}_{2}}{2} & \frac{\tilde{\kappa}_{1}}{3} & 1 & 0\\
\tilde{\kappa}_{5} & \frac{\tilde{\kappa}_{3}}{2} & \frac{\tilde{\kappa}_{2}}{3} & \frac{\tilde{\kappa}_{1}}{4} & 1\\
\vdots & \vdots & \vdots & \vdots &  & \ddots\\
 &  &  &  &  &  & 1\\
\tilde{\kappa}_{n} & \frac{\tilde{\kappa}_{n-2}}{2} & \frac{\tilde{\kappa}_{n-3}}{3} & \frac{\tilde{\kappa}_{n-4}}{4} &  &  & \frac{\tilde{\kappa}_{1}}{n-1}
\end{array}\right].
\end{align*}
Substituting (\ref{eq:kappatilde}) yields the result.

\end{proof}
Finally, notice that an expression of the cumulants as a sum over
partitions is given by \cite{Rota}
\[
\mu_{n}=\sum_{\pi\vdash n}\kappa_{\pi}
\]
with, for $\left(\pi_{1},\dots,\pi_{k}\right)$ a partition of $n,$
$\kappa_{\pi}=\prod_{j=1}^{k}\kappa_{\pi_{j}}.$

\section{A conjecture}

We look now at the equivalent of Romik's sequence $d\left(n\right)$ for
some special values of the elliptic modulus $k\ne\frac{1}{\sqrt{2}}.$ Recall that Corollary \ref{corollary} stated that $d\left(n\right)=2^{n}R_{4n}\left(\frac{1}{\sqrt{2}}\right), n\ge1$. Note that in the standard Romik case, we have $R_{2n+2}\left(\frac{1}{\sqrt{2}}\right)=0$ so that we consider $R_{4n}$, but in general we have $R_{2n+2}\left(\frac{1}{\sqrt{\ell}}\right)\neq 0$ so we must instead consider $R_{2n}$. The following table shows the sequences
$\alpha_{n}^{\left( \ell \right)} R_{2n}\left( \frac{1}{\sqrt{\ell}} \right)$ for all values $k=\frac{1}{\sqrt{\ell}}$ with $3 \le \ell \le 7,$ where $\alpha_{n}^{\left( \ell \right)}$ is a properly chosen prefactor. We conjecture that these sequences are integral for any $n \ge 1$. 
\begin{conj}
For any fixed integer value of $\ell \geq 3$, with the normalization factor
$$ \alpha_{n}^{\left( \ell \right)} =\frac{\ell^n}{2\ell-2},$$
we have the integrality result
$$ \alpha_{n}^{\left( \ell \right)} R_{2n}\left( \frac{1}{\sqrt{\ell}} \right) \in \mathbb{Z}, n\geq 1.$$
\end{conj}
Note that the resulting sequence is normalized so that the first term is $0$ and the second term is $1$, since we have the explicit result $R_2(k)=0, R_4(k) = 2(kk')^2=2k^2(1-k^2)$, so that $R_4\left(  \frac{1}{\sqrt{l}} \right) = \frac{2(\ell-1)}{\ell^2},$
which is precisely $\frac{1}{\alpha_2^{(\ell)}}$.
\begin{center}
\begin{tabular}{|c|c|c|}
\hline 
$k$ 
& sequence $\alpha_{n}^{\left( \ell \right)} R_{2n}\left( \frac{1}{\sqrt{\ell}} \right),$ $n \ge 1$\tabularnewline
\hline 
\hline 
$\frac{1}{\sqrt{3}}$ 
& $0, 1, 4, 12, -32, 112, 9024, 188992$\tabularnewline
\hline 
$\frac{1}{2}$ 
& $0,1,8,58,224,172,15008,929368$\tabularnewline
\hline 
$\frac{1}{\sqrt{5}}$ 
& $0,1,12,136,1152,6720,43776,2314752$\tabularnewline
\hline 
$\frac{1}{\sqrt{6}}$ 
& $0,1,16,246,3136,32236,280896,5779016$\tabularnewline
\hline 
$\frac{1}{\sqrt{7}}$ 
& $0,1,20,388,6560,95344,1194560,18811072$\tabularnewline
\hline 
\end{tabular}
\end{center}
We believe that stronger integrality results are possible; it appears that we can divide out increasing powers of $2$ from these sequences, which still leaves them integer valued. However, we have been unable to formulate a precise conjecture for which power of $2$ corresponds to which $\ell$. For instance, in the case $\ell=3$ it appears that we have
$$\frac{1}{2^{n-2}}\alpha_{n}^{\left( 3 \right)} R_{2n}\left( \frac{1}{\sqrt{3}} \right) \in \mathbb{Z}.$$
%
%
%
%
%
%
%
%
%
%
\vspace{0.5cm}
Numerically these sequences appear to (in some cases) display $p$-adic behavior mirroring $d(n)$: modulo some primes, the sequences $\alpha_{n}^{\left( \ell \right)} R_{2n}\left( \frac{1}{\sqrt{\ell}} \right)$ are either periodic or eventually vanishing. For instance, for $\ell=4,$ the residues modulo $2,3,5,7,11,13$ all apparently vanish for sufficiently large $n$. However, we have been unable to find a simple rule for which primes and values of $\ell$ lead to interesting behavior of $\alpha_{n}^{\left( \ell \right)} R_{2n}\left( \frac{1}{\sqrt{\ell}}\right)$ mod $p$. We leave it as an open problem to prove various congruences for these sequences. We also note that none of these sequences appears in the online encyclopedia OEIS.

\section*{Acknowlegments}

The authors thank Karl Dilcher for his unconditional support, Lin
Jiu for his constructive comments, and the Coburg Social Caf\'{e}
for the excellent vibes. Moreover, they thank the referee for his careful review and Officer G. who made all this possible.

\end{document}